\documentclass[a4paper, 11pt]{article} 

\usepackage[svgnames]{xcolor} 

\usepackage[utf8]{inputenc} 
\usepackage[T1]{fontenc} 
\usepackage{PTSerif} 
\usepackage[nottoc]{tocbibind}

\usepackage{hyperref}
\usepackage{amsmath}
\usepackage{amssymb}
\usepackage{amsthm}
\usepackage{blindtext}
\usepackage[margin = 1 in]{geometry}
\usepackage{mathtools}
\usepackage{graphicx}
\usepackage{tikz-cd}
\usepackage{footnote}
\usepackage{fancyhdr}
\usepackage{hyperref}
\usepackage{stmaryrd}
\usepackage{comment}
\usepackage[normalem]{ulem}
\hypersetup{
    colorlinks,
    linkcolor=blue,
    urlcolor=blue
}

\newtheorem{theorem}{Theorem}[section]

\newtheorem{proposition}[theorem]{Proposition}

\theoremstyle{definition}

\theoremstyle{remark}
\newtheorem*{remark}{Remark}

\graphicspath{ {images/} }

\newcommand{\mb}{\mathbb}

\newcommand{\Q}{\mb{Q}}
\newcommand{\R}{\mb{R}}
\newcommand{\Z}{\mb{Z}}
\newcommand{\C}{\mb{C}}

\newcommand{\F}{\mb{F}}

\renewcommand{\P}{\mb{P}}

\newcommand*{\sheafhom}{\mathcal{H}\kern -.5pt om}

\newcommand{\Car}{\underline{C_2}^+}

\DeclareMathOperator{\Gal}{Gal}

\DeclareMathOperator{\GL}{GL}

\usepackage{sectsty}
\title{A note on regular polyhedra over finite fields}
\author{Caleb Ji}
\date{\today}

\begin{document} 

\maketitle 
\begin{abstract}
    Grothendieck proposed a theory of regular polyhedra over finite fields in Section 4 of \textit{Esquisse d'un Programme}.  He isolates certain key parameters from the automorphism groups of regular polyhedra, which can be extended to any genus and specialized to various rings.  In this note we give an interpretation of his sketched theory which explains some of his observations.  We are able to compute some explicit examples and address a question Grothendieck raised about them in connection to dessins d'enfants.  Finally, we highlight some of Grothendieck's observations which remain unexplained by our current approach. 
\end{abstract}

\section{Introduction} 
\subsection{Historical background}
In order to solve the quintic equation, Klein studied the algebraic geometry of the Platonic solids by considering them as finite branched covers of $\C\P^1$ in his book \textit{Lectures on the Icosahedron}.  This point of view on regular polyhedra is known today as part of the theory of dessins d'enfants, following Sections 2 and 3 of Grothendieck's proposal \textit{Esquisse d'un Programme} \cite{esquisse}.  The new element that Grothendieck adds is the action of the absolute Galois group $\Gal(\overline{\Q}/\Q)$ on the dessins, which led to his reflections on anabelian geometry.  Since then, much work has been done both on calculations with dessins and in anabelian geometry.  However, this present note is concerned with Grothendieck's new point of view on regular polyhedra over finite fields, which is the subject of Section 4 of \cite{esquisse}.  This topic can be viewed as the study of dessins d'enfants that correspond to Galois coverings of $\P^1-\{0, 1, \infty\}$.  The Galois condition restricts the combinatorial nature of the dessin considered and leads to a natural generalization of regular polyhedra to any base ring.  \\ 

Despite giving an outline of the theory and raising some concrete questions, this part of Grothendieck's proposal has not seemed to attract any attention.  One possible reason for this is that some of the results he states seem difficult to properly understand and reproduce.  The purpose of this note is to give one plausible interpretation of his theory and compute some examples with it.  Though this interpretation does not explain all of Grothendieck's observations in this  section, we hope that it will serve as an advertisement for others who may be interested in pursuing these ideas further. \\ 

In the remainder of this section we will give a quick overview of dessins d'enfants.  The purpose of Section 2 is to explain Grothendieck's definition of the universal regular polyhedron and his method of defining regular polyhedra over arbitrary base rings in detail.  In Section 3 we will present some explicit calculations regarding regular polyhedra over finite fields, motivated by a question Grothendieck raised concerning which finite regular maps can be obtained from regular polyhedra over finite fields. \\ 

\textbf{Acknowledgements.} We would like to thank Will Sawin for some interesting comments. 

\subsection{Review of dessins d'enfants} 
Here we give a brief introduction to dessins d'enfants, focusing only on what we need to study the Galois case.  More comprehensive accounts can be found in e.g. \cite{schneps}. \\  

A dessin d'enfant is given by a surface $X$ with a graph drawn on it that divides the surface into open cells.  A flag is defined to be the set consisting of a vertex, the midpoint of an edge containing the vertex, and the center of a face containing the edge.  Following \cite{esquisse}, define the cartographic group $\uline{C}_2$ by 
\[
\uline{C}_2 \coloneqq \langle \sigma_v, \sigma_e, \sigma_f | \sigma_v^2=\sigma_e^2=\sigma_f^2=(\sigma_v\sigma_f)^2=1 \rangle.
\]

Note that $\uline{C}_2$ acts on the set of flags of $X$ by allowing $\sigma_v, \sigma_e, \sigma_f$ to act by reflections of vertices, edges, and faces respectively. \\

It will be useful to consider an oriented version, where we consider only the oriented flags, namely the ones who in counterclockwise order are given by face -- vertex -- edge. Now define the oriented cartographic group by 
\[
\underline{C}_2^+ \coloneqq \langle \rho_v, \rho_e, \rho_f | \rho_f\rho_e\rho_v= \rho_e^2=1 \rangle.
\] 
Note that $\Car$ is the index 2 subgroup of $\uline{C}_2$ generated by $\rho_v = \sigma_f\sigma_e, \rho_e = \sigma_v\sigma_f, \rho_0=\sigma_e\sigma_v$.  Then $\Car$ acts on the set of oriented flags of a dessin by allowing the elements $\rho_v, \rho_e, \rho_f$ to correspond to the counterclockwise rotation of the flag around the vertex, edge, and face, respectively. \\ 

Every dessin is determined by the action of the oriented cartographic group on one of its flags.  Thus a dessin corresponds to a quotient of $\Car$.  Because $\Car$ is isomorphic to the fundamental group of $\P^1$ minus 3 points with the additional relation $\rho_1^2=1$, we see that the dessin in fact corresponds to a finite cover of $\P^1$ branched only at $0, 1, \infty$ with ramification order 1 or 2 above 1.  Note that by composing with the morphism $f:\P^1\rightarrow \P^1$ defined by $f(z) = 4z(1-z)$, we can impose the latter condition on any curve satisfying the former.  Thus by Belyi's theorem, an algebraic curve can be defined over $\overline{\Q}$ if and only if it can be written in this way.   

\subsection{The Galois case} 
We now focus on the case when the dessin gives a Galois covering of $\P^1-\{0, 1, \infty\}$, which means that the automorphism group of the graph acts simply transitively on the set of flags.  These dessins correspond to the finite quotients of $\Car$.  However, as Grothendieck notes, it is important to actually consider all quotients of $\Car$.  We will henceforth refer to `regular polyhedra' in this generalized sense. \\ 

Let us begin with some examples.  For each pair of positive integers $(p, q)$, we may consider the dessin associated to the quotient 
\[
G_{p, q} = \underline{C}_2^+/\rho_0^p=\rho_2^q=1. 
\] 
The corresponding dessin may be drawn by beginning with a flag and transforming it according to all elements of $G_{p, q}$.  As long as $p, q\ge 3$, the newly imposed conditions imply that we obtain a regular tiling with $p$ faces to a vertex and $q$ edges to a face.  Note that $G_{p, q}$ is finite if and only if $p\cdot \dfrac{(q-2)\pi}{q}<2\pi$.  These cases correspond precisely to the classical regular polyhedra, shown in Table \ref{finite}. \\ 

\begin{table}[h!]
    \centering
    \begin{tabular}{c|c}
       (p,q)  &  Polyhedron \\
       \hline
    (3, 3) & tetrahedron \\
    (3, 4)  & cube \\
    (4, 3)  &  octahedron \\
    (3, 5)  &  dodecahedron \\
    (3, 5)  &  icosahedron \\
    \end{tabular}
    \caption{The classical regular polyhedra}
    \label{finite}
\end{table}

If $p\cdot \dfrac{(q-2)\pi}{q}=2\pi$, then we obtain a regular tiling of the Euclidean plane.  The three options here are $(p,q) = (3, 6)$ (hexagons), $(4, 4)$ (squares), and $(6, 3)$ (triangles). \\ 

Apart from these two finite lists of cases, we obtain regular tilings of the hyperbolic plane. \\ 

The groups $G_{p, q}$ do not exhaust all possible quotients of $F_2$, whether we restrict to the finite ones or not.  However, as Grothendieck pointed out in Section 4 of \cite{esquisse}, one can realize the infinite regular polyhedra as covers of finite ones in a natural way.  Namely, one defines the `universal regular 2-polyhedron' via two parameters that all regular polyhedra satisfy, and by reducing these parameters modulo $p$ one obtains finite regular polyhedra defined over $\F_p$.  This approach also allows one to define regular polyhedra over any base ring.  One question Grothendieck raised is whether all finite regular polyhedra, or equivalently algebraic curves over $\overline{\Q}$, can be obtained through regular polyhedra over finite fields.  

\section{The universal regular polyhedron and specialization} 
To define the universal 2-polyhedron, Grothendieck says the following in Section 4 of \cite{esquisse}. \\ 

``For such a polyhedron, we find a canonical basis (or flag) of the ambient affine or projective space, such that the operations of the cartographic group $\underline{C}_n$, generated by the fundamental reflections $\sigma_i (0 \le i \le n)$, are written in that basis by universal formulae, in terms of
the $n$ parameters $\alpha_1, \ldots, \alpha_n$, which can be geometrically interpreted as the
doubles of the cosines of the ``fundamental angles" of the polyhedron." \\ 

We will now determine the formulas for $n=2$.  Given any regular polyhedron, let $\theta$ be the angle of each face and let $\gamma$ be the dihedral angle between two faces.  Fix a basis $v_0, v_1, v_2$ given by the vertex, midpoint of an edge, and center of a face of a flag.  Let $\sigma_v, \sigma_e, \sigma_f$ represent reflection of a vertex, edge, and face respectively.  Then by elementary geometry, we obtain:   
\begin{align*}
    \sigma_0(v_0) &= 2v_1-v_0, \\
    \sigma_1(v_1)&=(1-\cos\theta)v_0 - v_1 + (1+\cos\theta)v_2, \\
    \sigma_2(v_2) &= (1-\cos\gamma)v_1-v_2.  
\end{align*}

\begin{remark}
We note that these formulas are in terms of the cosines of the angles of the polyhedrons, not their doubles as Grothendieck states.  We have been unable to resolve this discrepancy.
\end{remark}

Then we obtain 
\[ 
\rho_v= \sigma_2\circ\sigma_1 = 
\begin{bmatrix}
1 & 1-\cos\theta & (1-\cos\gamma)(1-\cos\theta) \\
0 & -1 & -1+\cos\gamma \\ 
0 & 1+\cos\theta & -1 + (1-\cos\gamma)(1+\cos\theta) 
\end{bmatrix},
\] 

\[ 
\rho_e= \sigma_2\circ\sigma_0 = 
\begin{bmatrix}
-1 & 0 & 0 \\
2 & 1 & 1-\cos\gamma \\ 
0 & 0 & -1 
\end{bmatrix},
\]

\[ 
\rho_f= \sigma_1\circ\sigma_0 = 
\begin{bmatrix}
-1 & -1+\cos\theta & 0 \\
2 & 1 - 2\cos\theta  & 0 \\ 
0  & 1+\cos\theta & 1 
\end{bmatrix}.
\]

In order to interpret regular polyhedra over arbitrary base rings, Grothendieck isolates these universal formulae in terms of $\cos\theta$ and $\cos\gamma$ as the key to giving a sound definition of regular polyhedra.  He goes on to say: \\ 

``In this game, there is no question of limiting oneself to
finite regular polyhedra, nor even to regular polyhedra whose facets are of
finite order, i.e. for which the parameters $\alpha_i$ are roots of suitable ``semicyclotomic" equations, expressing the fact that the ``fundamental angles" (in
the case where the base field is $\R$) are commensurable with $2\pi$." \\

Thus for any ring $R$, the regular polyhedra over $R$ are defined through the above formulas where elements in $R$ are substituted for $\cos\theta$ and $\cos\gamma$.  We see that when $R=\R$, each of the groups $G_{p, q}$ considered above are obtained through particular choices of $\cos\theta$ and $\cos\gamma$. \\ 

This approach also allows us to define specialization of regular polyhedra.  This is explained in Section 4 of \cite{esquisse} as follows. \\ 

``The situation is entirely different if we start
from an infinite regular polyhedron, over a field such as $\Q$, for instance, and
''specialise" it to the prime fields $\F_p$ (a well-defined operation for all $p$ except
a finite number of primes). It is clear that every regular polyhedron over a
finite field is finite – we thus find an infinity of finite regular polyhedra as
$p$ varies, whose combinatorial type, or equivalently, whose automorphism
group varies “arithmetically” with $p$. This situation is particularly intriguing in the case where $n = 2$, where we can use the relation made explicit
in the preceding paragraph between combinatorial 2-maps and algebraic
curves defined over number fields. In this case, an infinite regular polyhedron defined over any infinite field (and therefore, over a sub-$\Z$-algebra of it
with two generators) thus gives rise to an infinity of algebraic curves defined
over number fields, which are Galois coverings ramified only over 0, 1 and
$\infty$ of the standard projective line." \\ 

Indeed, by considering the matrices defining $\rho_v, \rho_e, \rho_f$ to be in $GL(3, R)$, we obtain various quotients of $F_2$, which are necessarily finite if $R$ is finite.  

\section{Regular polyhedra over finite fields} 
Grothendieck raised the following question in Section 4 of \cite{esquisse}:
``Exactly
which are the finite regular 2-maps, or equivalently, the finite quotients of the 2-cartographic group, which come from regular 2-polyhedra over finite
fields? Do we obtain them all, and if yes: how?" \\ 

In this section we will make some computations towards answering this question and at the same time give some examples of specialization.  Because a single algebraic curve may be realized as a finite regular 2-map in multiple ways, it is also interesting to ask the same question where we are only interested in the isomorphism class of the curve in question. \\ 

Since regular polyhedra are given by finite quotients of $\Car$, we can identify them by finite quotients of the group with presentation $\langle \rho_v, \rho_f | \rho_v^p=\rho_f^q = \rho_v\rho_f\rho_v\rho_f = 1 \rangle$ where $\rho_v$ and $\rho_f$ have order $p$ and $q$ respectively.  One can use the Riemann-Hurwitz formula to compute the genus of the associated algebraic curve $X$ in terms of $p$, $q$, and the order of the group.  Indeed, say the map has $V$ vertices, $E$ edges, and $F$ faces.  If the order of the group is $|G|$, then we obtain a degree $|G|=2E$ map from $X$ to $\P^1$ ramified at 0, 1, and $\infty$.  The total ramification order over these points is given by $3|G| - V-E-F$.  On the other hand, we have $E = \frac{pV}{2} = \frac{qF}{2}$.  Therefore we have 
\[
2g_X - 2 = -2|G| + 3|G| - V - E - F = E\left(1 - \frac{2}{p} - \frac{2}{q}\right).
\]

We will now give some examples for $g=0, 1$.  

\subsection{Genus 0} 
If $g=0$, then it is easy to classify all finite regular polyhedra.  Indeed, this was already done in Table \ref{finite}.  To include the case when either $p$ or $q$ is less than 3, we also have the cases $(p, q, E)=(2, n, n), (n, 2, n)$ for $n\in \Z_{>0}$. \\ 

We will now list the universal formulas for the five classical regular polyhedra and consider their reductions to finite fields.  Looking at the values of the cosines, this operation is well-defined except when $p=2, 3$ for the tetrahedron and octahedron, always for the cube, except $p=2,3,5$ for the dodecahedron, and except $p=2,5$ for the icosahedron. \\ 

Setting $x = \cos\theta, y = \cos\gamma$, we recall that the formulas for the rotations are given by 
\[
\rho_v= 
\begin{bmatrix}
1 & 1-x & (1-x)(1-y) \\
0 & -1 & -1+y \\ 
0 & 1+x & -1 + (1+x)(1-y) 
\end{bmatrix},
\rho_e= 
\begin{bmatrix}
-1 & 0 & 0 \\
2 & 1 & 1-y \\ 
0 & 0 & -1 
\end{bmatrix},
\rho_f= 
\begin{bmatrix}
-1 & -1+x & 0 \\
2 & 1 - 2x  & 0 \\ 
0  & 1+x & 1 
\end{bmatrix}.
\] 

The values of $\cos\theta$ and $\cos\gamma$ for the regular polyhedra are given in the following table. 

\begin{table}[h]
    \centering
    \begin{tabular}{c|c|c}
    Polyhedron &  $\cos\theta$ & $\cos\gamma $ \\
    \hline
    tetrahedron & $\frac{1}{2}$ & $\frac{1}{3}$ \\
    cube & 0 & 0 \\
    octahedron & $\frac{1}{2}$ & $-\frac{1}{3}$ \\
    dodecahedron & $\frac{1-\sqrt{5}}{4}$ & $-\frac{\sqrt{5}}{5}$ \\
    icosahedron & $\frac{1}{2}$ & $-\frac{\sqrt{5}}{3}$  \\
    \end{tabular}
    \caption{Angles of the classical regular polyhedra}
    \label{angles}
\end{table}

\begin{proposition}
The classical regular polyhedra retain their automorphism groups when reduced to $\F_p$ when this reduction is defined, except that the cube in $\F_2$ has automorphism group $S_3$. 
\end{proposition} 
\begin{proof}
The automorphism group of the tetrahedron is $A_4$.  Its normal subgroups are the trivial one, one of order 4, and itself.  Therefore, if its specialization at some prime $p$ can only change if the automorphism group becomes trivial or of order 3.  For each $p$, it is clear that $\rho_v$ and $\rho_f$ do not reduce to the identity and are distinct, so the result follows. 

The normal subgroups of $S_4$ are the trivial one, one of order 4, $A_4$, and itself.  One checks that for the cube in $\F_2$, the element $\rho_v$ still has order 3 but $\rho_f$ has order 2.  This forces the automorphism group to be the quotient of $S_4$ by its normal subgroup of order 4, which is isomorphic to $S_3$.  In other characteristics $\rho_f$ still has order 4 and $\rho_v$ has order 3, so we obtain $S_4$.  This last statement also holds for the octahedron in characteristics larger than 3. 

The automorphism group of the dodecahedron and icosahedron are $A_5$, which is simple.  It is easy to check that $\rho_e$ and $\rho_f$ never reduce to the identity modulo $p$, so their automorphism groups remain $A_5$. 
\end{proof} 
\begin{remark}
    This confirms Grothendieck's statement in \cite{esquisse} that an icosahedron specialized to a finite field remains an icosahedron.  However, he seems to claim this in all cases. 
\end{remark}

\subsection{Genus 1}  
Every finite regular polyhedron of genus 1 must satisfy $(p, q)=(3, 6), (4, 4)$, or $(6, 3)$ where as before $p$ and $q$ are the orders of the elements $\rho_v$ and $\rho_f$ respectively.  Taking into account the isomorphism between $G_{3, 6}$ and $G_{6, 3}$, we see they can all be realized as finite quotients of $G_{4, 4}=\Car/\rho_v^4=\rho_f^4=1, G_{6, 3} = \Car/\rho_v^6 = \rho_f^3=1$ modulo $p$.  We begin by giving an explicit example of specializing these infinite polyhedra to finite base rings.

\begin{proposition}
    For $n\ge 2$, reducing $G_{4, 4}$ modulo $n$ gives the map defined by an $n/2\times n/2$ square grid if $n$ is even and an $n\times n$ square grid if $n$ is odd, which in both cases correspond to the elliptic curve defined by $y^2=x^3+x$, with lattice $\langle 1, i\rangle $ and $j$-invariant $1728$. 
\end{proposition}
\begin{proof}
    In the case of $G_{4, 4}$ we have $\cos\theta=0, \cos\gamma=-1$, so the matrices of interest are given by 
    \[
\rho_v= 
\begin{bmatrix}
1 & 1 & 2 \\
0 & -1 & -2 \\ 
0 & 1 & 1 
\end{bmatrix},
\rho_f= 
\begin{bmatrix}
-1 & -1 & 0 \\
2 & 1  & 0 \\ 
0  & 1 & 1 
\end{bmatrix}. 
    \] 

    Taken in $\GL(3, \Z)$ these two matrices generate $G_{4, 4}$, which is represented by an infinite square grid which we may identify with the coordinate plane.  Reducing modulo $n$ is equivalent to considering their image in $\GL(3, \Z/n\Z)$.  The corresponding map can thus be computed by analyzing which transformations of the flag are represented by a matrix equivalent to the identity in $\GL(3, \Z/n\Z)$, and quotienting out by the subgroup they generate.   

Starting with the flag with coordinates $V = (1, 0), E = (1/2, 0), F = (1/2, 1/2)$, we see that the matrices $I_3+R$ and $I_3+U$, where  
\[
R = \begin{bmatrix}
-2 & -2 & -2 \\
2 & 2 & 2 \\ 
0 & 0 & 0 
\end{bmatrix}, 
U = \begin{bmatrix}
0 & 0 & 0 \\
-2 & -2 & -2 \\ 
2 & 2 & 2 
\end{bmatrix}
\] 
send it to the corresponding flag in the square to the right and to the square above, respectively.  On the other hand, the rotations are represented by the matrices
\[
\begin{bmatrix}
-1 & -1 & 0 \\
2 & 1 & 0 \\ 
0 & 0 & 1 
\end{bmatrix}, 
\begin{bmatrix}
-1 & 0 & 0 \\
0 & -1 & 0 \\ 
2 & 2 & 1 
\end{bmatrix}, 
\begin{bmatrix}
1 & 1 & 0 \\
-2 & -1 & 0 \\ 
2 & 1 & 1 
\end{bmatrix}.
\] 

One also sees that the matrices corresponding to shifting the flag to the square above and to the right for each of these three other flags in the same square are of the same form as $R$ and $U$.  One deduces from this that the only transformations equivalent to the identity are those which move the original flag to the same flag a multiple of $n/2$ squares horizontally and vertically for $n$ even, and a multiple of $n$ squares for $n$ odd. \\ 

The elliptic curve associated to this dessin is well-known to be the one with $j$-invariant 1728 and lattice given by $\langle 1, i\rangle$; see e.g. \cite{frechette}.  In the case of a 1 by 1 square, an explicit Belyi function is given by $f(z) = \frac{c}{\wp_i^2(z)}$, where $c$ is a constant used to ensure that the function is 1 at the midpoints of the edges and $\wp_i$ is the Weierstrass $\wp$-function on the lattice $\langle 1, i\rangle$.  In the case of a $k$ by $k$ square, we can simply replace this function by $f(z) = \frac{c}{\wp_i^2(kz)}$.
\end{proof}

\begin{remark}
    In this example, if we restrict $n$ to being equal to a prime $p$ we clearly do not obtain every possible map because the degree of each map will then be of the form $p^2$.  However, order considerations do not rule out the possibility that we can obtain them using arbitrary regular polyhedra over finite fields.  Indeed, since the order of $GL(3, \F_q)$ is divisible by any $n$ for an appropriate choice of $q$, there is no obvious restriction on what order regular polyhedra obtained in such a way may have. \\ 

    For an example of a map that can be obtained as a quotient of $G_{4, 4}$ not of such a simple form, see \cite{sijsling}, Figure 3.7. 
\end{remark}

One can use a similar approach to reducing $G_{6, 3}$ modulo $n$ and obtain a triangular grid rather than a square grid.  In this case, one can use the meromorphic function $f(z) = \frac{c}{\wp_{e^{i\pi/3}}(z)^3}$, scaled appropriately, to exhibit an explicit Belyi function on the elliptic curve with $j$-invariant 0 and lattice given by $\langle 1, e^{i\pi/3}\rangle$.

\subsection{Additional remarks} 
Under the interpretation we have taken, one may try to complete the classification of which regular maps and algebraic curves can be obtained over finite fields in genus 1.  In higher genus one can ask the same question, which seems to be more difficult.  However, another line of work involves fully understanding the observations of Grothendieck in Section 4 of \cite{esquisse}.  Indeed, in this note we have pointed out some instances where our interpretation does not account for everything outlined by Grothendieck.  In particular, we highlight again the fact that our universal formulas defining the reflections and rotations are defined over $\Z[\cos\theta, \cos\gamma]$, rather than using the doubles of the cosines as Grothendieck states.  For an octahedron we have $\cos\theta = \frac{1}{2}$, so under our conventions it is not clear how to consider the octahedron over $\F_2$.  This point seems significant because Grothendieck describes the octahedron over $\F_2$ as part of an interesting phenomenon, which he describes as follows. \\ 

``Thus, examining the Pythagorean polyhedra
one after the other, I saw that the same small miracle was repeated each
time, which I called the \underline{combinatorial paradigm} of the polyhedra under
consideration. Roughly speaking, it can be described by saying that when
we consider the specialisation of the polyhedra in the or one of the most singular characteristic(s) (namely characteristics 2 and 5 for the icosahedron, characteristic 2 for the octahedron), we read off from the geometric
regular polyhedron over the finite field ($\F_4$ and $\F_5$ for the icosahedron, $\F_2$
for the octahedron) a particularly elegant (and unexpected) description of
the combinatorics of the polyhedron." \\  

One possible approach may be to use a choice of basis different from the canonical one given by a flag.  This may allow, say, the octahedron to be reduced to $\F_2$, though it appears that the application of this method would be specific to the polyhedron in question.  Another approach may be to consider the polyhedron in projective space, which Grothendieck indeed mentioned in \cite{esquisse} as a superior alternative to considering them in affine space.  We conclude that though the approach described in this note explain some of Grothendieck's ideas on regular polyhedra, some variations will likely be necessary to realize them fully.

\end{document}